\documentclass[10pt]{amsart}
\usepackage{amsmath}
\usepackage{amssymb, eufrak}%
\usepackage{amsthm}
\usepackage[numbers, square]{natbib}

\theoremstyle{plain}
\newtheorem{theorem}{Theorem}[]
\newtheorem{lemma}{Lemma}[]
\newtheorem{proposition}{Proposition}[]

\theoremstyle{definition}
\newtheorem{definition}{Definition}[]

\theoremstyle{remark}
\newtheorem{remark}{Remark}[]

\newcommand{\eqdistr}{\stackrel{d}{=}}
\renewcommand{\P}{\mathbb P}
\newcommand{\SSS}{\mathbb S}
\newcommand{\W}{\mathbb W}
\newcommand{\U}{\mathbb U}
\newcommand{\E}{\mathbb E}
\newcommand{\R}{\mathbb{R}}
\newcommand{\N}{\mathbb{N}}

\newcommand{\Z}{\mathbb{Z}}

\newcommand{\XXX}{\mathbb{X}}

\newcommand{\Cov}{\mathrm{Cov}}

\newcommand{\ka}{\kappa}

\newcommand{\Rect}{\boldsymbol{\mathfrak{R}}^d}
\newcommand{\RectDim}{\boldsymbol{\mathfrak{R}}^{d-1}}
\newcommand{\CRect}{\boldsymbol{\mathcal{R}}^d}
\newcommand{\DInt}{\boldsymbol{\mathcal{D}}}
\newcommand{\CRectm}{\boldsymbol{\mathcal{P}}}

\newcommand{\CRectmm}{\boldsymbol{\mathcal{Q}}}
\newcommand{\Cube}{\boldsymbol{\mathfrak{C}}^d}
\newcommand{\CCube}{\boldsymbol{\mathcal{C}}^d}
\newcommand{\Pcoll}{\boldsymbol{\mathfrak{P}}^d}
\newcommand{\Pcollone}{\boldsymbol{\mathfrak{P}}^1}
\newcommand{\Pcollm}{\boldsymbol{\mathfrak{Q}}^d}
\newcommand{\BB}{\mathcal{B}}

\newcommand{\RR}{\boldsymbol{\mathcal{I}}}
\newcommand{\DRR}{\mathfrak{I}}

\newcommand{\xx}{\boldsymbol{x}}
\newcommand{\hh}{\boldsymbol{h}}
\newcommand{\yy}{\boldsymbol{y}}
\newcommand{\zz}{\boldsymbol{z}}
\newcommand{\kk}{\boldsymbol{k}}
\newcommand{\lbold}{\boldsymbol{l}}
\newcommand{\pp}{\boldsymbol{p}}
\newcommand{\qq}{\boldsymbol{q}}
\begin{document}

\title{Extremes of the standardized Gaussian noise}
\author{Zakhar Kabluchko}
\address{Institute of Stochastics, Ulm University, Helmholtzstr.\ 18, 89069 Ulm, Germany}

\begin{abstract}
Let $\{\xi_n, n\in\Z^d\}$ be a $d$-dimensional array of i.i.d.\ Gaussian random variables and define $\SSS(A)=\sum_{n\in A} \xi_n$, where $A$ is a finite subset of $\Z^d$. We prove that the appropriately normalized maximum of $\SSS(A)/\sqrt{|A|}$, where $A$ ranges over all discrete cubes or rectangles contained in $\{1,\ldots,n\}^d$, converges in the weak sense to the Gumbel extreme-value distribution as $n\to\infty$. We also prove continuous-time counterparts of these results.
\end{abstract}
\maketitle

\section{Introduction and statement of results}\label{sec:main}
Let $\{\xi_i, i\in\N\}$ be independent standard Gaussian random variables. Denote by $S_k=\xi_1+\ldots+\xi_k$ the corresponding random walk and let
\begin{equation}
L_n=\max_{0\leq i<j\leq n}\frac{S_j-S_i}{\sqrt{j-i}}.
\end{equation}
It has been shown by~\citet{siegmund_venkatraman} that for every $\tau\in\R$,
\begin{equation}\label{eq:siegmund_venkatraman}
\lim_{n\to\infty}\P\left [L_n\leq \sqrt{2\log n}+\frac {\frac12 \log\log n+\log \frac{H}{2\sqrt{\pi}}+\tau}{\sqrt{2 \log n}} \right ]=e^{-e^{-\tau}},
\end{equation}
where $H\in(0,\infty)$ is some constant.
A different proof of the same result has been given in~\cite{kabluchko_stand_gauss} where also the following continuous-time counterpart of~\eqref{eq:siegmund_venkatraman} can be found. Let $\{B(t),t\geq 0\}$ be a standard Brownian motion. For $n>1$ define
\begin{equation}
M_n=\sup_{\genfrac{}{}{0pt}{1} {x,y\in[0,n]} {y-x\geq 1}} \frac {B(y)-B(x)}{\sqrt{y-x}}.
\end{equation}
Then, for every $\tau\in\R$,
\begin{equation}\label{eq:siegmund_venkatraman_cont}
\lim_{n\to\infty}\P\left [M_n\leq \sqrt{2\log n}+\frac {\frac 32 \log\log n-\log (2\sqrt \pi)+\tau}{\sqrt{2 \log n}}\right]=e^{-e^{-\tau}}.
\end{equation}
Almost sure laws of large numbers for $L_n$, $M_n$ and related quantities have been obtained in~\cite{shao}, \cite{steinebach}, \cite{kabluchko_munk2}.

Our aim here is to prove multidimensional counterparts of~\eqref{eq:siegmund_venkatraman} and~\eqref{eq:siegmund_venkatraman_cont}. We will be interested in the maximum of discrete- or continuous-time $d$-dimensional Gaussian noise standardized by the square root of its variance. The maximum is taken over some family of $d$-dimensional subsets. Here, we will consider two families of subsets, rectangles and cubes, in discrete and continuous setting. Both families are multidimensional generalizations of the collection of one-dimensional intervals.

%

Let us state our discrete-time results first. Let $\{\xi_n, n\in\Z^d\}$ be a $d$-dimensional array of i.i.d.\ Gaussian random variables. Given a finite set $A\subset \Z^d$ we define
\begin{equation}
\SSS(A)=\sum_{n\in A} \xi_n.
\end{equation}
A set of the form $\{x_1,\ldots, x_1+h\}\times\ldots\times\{x_d,\ldots,x_d+h\}$, where $x_1,\ldots,x_d\in\Z$ and $h\in\N\cup\{0\}$, is called a $d$-dimensional discrete cube.  Denote by $\Cube$ the set of all discrete $d$-dimensional cubes and let $\Cube_n$ be the set of all discrete $d$-dimensional cubes contained in $\{1,\ldots,n\}^d$. Define
\begin{equation}\label{eq:def_ab_dcube}
u_n(\tau)=\sqrt{2d\log n}+\frac {\frac{1}{2}\log(d\log n)+\log \frac{(2d)^dJ_d}{\sqrt {\pi}}+\tau}{\sqrt{2d \log n}},\;\;\;\tau\in\R,
\end{equation}
where $J_d\in(0,\infty)$ is a constant defined in Lemma~\ref{lem:JG_finite} below.
\begin{theorem}\label{theo:main_dcube}
For every $\tau\in\R$,
$$
\lim_{n\to\infty}\P\left [\max_{A\in\Cube_n}\frac{\SSS(A)}{\sqrt{|A|}}\leq u_n(\tau) \right ]=e^{-e^{-\tau}}.
$$
\end{theorem}

A set of the form $\{x_1,\ldots, y_1\}\times\ldots\times \{x_1,\ldots,y_d\}$, where $x_i,y_i\in\Z$ and $x_i\leq y_i$ for all $1\leq i\leq d$, is called a $d$-dimensional discrete rectangle. Note that a discrete cube is a discrete rectangle whose sides have equal lengths. Denote by $\Rect$ the collection of all discrete $d$-dimensional rectangles and let $\Rect_n$ be the set of all discrete $d$-dimensional rectangles contained in $\{1,\ldots,n\}^d$.
Define
\begin{equation}\label{eq:def_ab_drect}
u_n(\tau)=\sqrt{2d\log n}+\frac {\left(d-\frac{1}{2}\right)\log(d\log n)+\log\frac{2^{2d-1}d^dG_d^d}{\sqrt {\pi}}+\tau}{\sqrt{2d \log n}},\;\;\;\tau\in\R,
\end{equation}
where $G_d\in(0,\infty)$ is a constant defined in Lemma~\ref{lem:JG_finite} below.
\begin{theorem}\label{theo:main_drect}
For every $\tau\in\R$,
$$
\lim_{n\to\infty}\P\left [\max_{A\in\Rect_n}\frac{\SSS(A)}{\sqrt{|A|}}\leq u_n(\tau) \right ]=e^{-e^{-\tau}}.
$$
\end{theorem}
\begin{remark}
The following laws of large numbers hold, see~\cite{kabluchko_munk1}:
$$
\lim_{n\to\infty}\frac{1}{\sqrt{2d\log n}}\max_{A\in\Cube_n}\frac{\SSS(A)}{\sqrt{|A|}}
=\lim_{n\to\infty}\frac{1}{\sqrt{2d\log n}}\max_{A\in\Rect_n}\frac{\SSS(A)}{\sqrt{|A|}}=1 \;\;\; \text{a.s.}
$$
\end{remark}
\begin{remark}
In dimension $d=1$, both Theorems~\ref{theo:main_dcube} and~\ref{theo:main_drect} reduce to~\eqref{eq:siegmund_venkatraman}.
\end{remark}

We also prove the following continuous-time counterparts of Theorems~\ref{theo:main_dcube} and~\ref{theo:main_drect}. Let $\{\W(A),A\in\BB(\R^d)\}$ be an independently scattered random Gaussian measure (white noise) on $\R^d$ whose intensity is the Lebesgue measure. This means that we are given a zero-mean Gaussian process $\W$ indexed by the collection $\BB(\R^d)$ of all Borel subsets of $\R^d$ such that for every $A_1,A_2\in\BB(\R^d)$,
$$
\Cov(\W(A_1),\W(A_2))=|A_1\cap A_2|, 
$$
where $|A|$ denotes the $d$-dimensional Lebesgue measure of a set $A\in\BB(\R^d)$.

A set of the form $[x_1,x_1+h]\times\ldots\times[x_d,x_d+h]$, where $x_1,\ldots,x_d\in\R$ and $h>0$ is called a $d$-dimensional cube.   Let $\CCube$ be the collection of all $d$-dimensional cubes and denote by $\CCube_n$ the set of all cubes contained in $[0,n]^d$. Endow $\CCube$ with its natural topology inherited from the identification $\CCube=\R^d\times (0,\infty)$. It is well-known that the process $\{\W(A), A\in\CCube\}$ has a version with a.s.\ continuous paths. This may be deduced for example from the continuity of the Brownian sheet process. In the sequel, we always deal with such a continuous version. It is not difficult to see that the supremum of the standardized white noise $\W(A)/\sqrt{|A|}$ taken over $A\in \CCube_n$ does not exist due to the singularity appearing as the volume of $A$ approaches $0$. To avoid the singularity, we take some $a> 0$ and define $\CCube_n(a)$ to be the set of all cubes from $\CCube_n$ whose side length $h$ satisfies $h\geq a$.
With a constant $E_d$ to be specified below, see~\eqref{eq:def_Ed}, define
\begin{equation}\label{eq:def_ab_ccube}
u_n(\tau)=\sqrt{2d\log n}+\frac{\left(d+\frac12\right)\log(d\log n)+\log\frac{2^dE_d}{da^d\sqrt{\pi}}+\tau}{\sqrt{2d \log n}},\;\;\;\tau\in\R.
\end{equation}
\begin{theorem}\label{theo:main_ccube}
For every $\tau\in\R$,
$$
\lim_{n\to\infty}\P\left [\sup_{A\in \CCube_n(a)}\frac{\W(A)}{\sqrt{|A|}} \leq u_n(\tau) \right ]=e^{-e^{-\tau}}.
$$
\end{theorem}

A set of the form $[x_1,y_1]\times\ldots\times[x_d,y_d]$, where $x_i,y_i\in\R$ and $x_i<y_i$ for all $1\leq i\leq d$ is called a $d$-dimensional rectangle. Note that a $d$-dimensional cube as a $d$-dimensional rectangle with equal side lengths. Let $\CRect$ be the collection of all rectangles.   We denote by $\CRect_n$ the set of all rectangles contained in $[0,n]^d$. Let $\{\W(A), A\in \BB(\R^d)\}$ be a white noise on $\R^d$. The random field $\{\W(A), A\in \CRect\}$ has a version with a.s.\ continuous paths. Given $a>0$ we define  $\CRect_n(a)$ to be the set of all rectangles $[x_1,y_1]\times\ldots\times[x_d,y_d]$ contained in $[0,n]^d$ such that $y_i-x_i\geq a$ for all $1\leq i\leq d$. We set
\begin{equation}\label{eq:def_ab_crect}
u_n(\tau)=\sqrt{2d\log n}+\frac {\left(2d-\frac{1}{2}\right)\log(d\log n)-\log (2 a^d\sqrt{\pi})+\tau}{\sqrt{2d \log n}},\;\;\;\tau\in\R.
\end{equation}
\begin{theorem}\label{theo:main_crect}
For every $\tau\in\R$,
$$
\lim_{n\to\infty}\P\left [\sup_{A\in \CRect_n(a)}\frac{\W(A)}{\sqrt{|A|}} \leq u_n(\tau) \right ]=e^{-e^{-\tau}}.
$$
\end{theorem}
\begin{remark}
Using methods similar to that of~\cite{kabluchko_munk1} it is possible to prove the following laws of large numbers: for every $a>0$,
$$
\lim_{n\to\infty}\frac{1}{\sqrt{2d\log n}}\sup_{A\in\CCube_n(a)}\frac{\W(A)}{\sqrt{|A|}}
=\lim_{n\to\infty}\frac{1}{\sqrt{2d\log n}}\sup_{A\in\CRect_n(a)}\frac{\W(A)}{\sqrt{|A|}}=1 \;\;\; \text{a.s.}
$$
\end{remark}
\begin{remark}
In dimension $d=1$, both Theorems~\ref{theo:main_ccube} and~\ref{theo:main_crect} reduce to~\eqref{eq:siegmund_venkatraman_cont}.
\end{remark}
The maxima of the standardized Gaussian noise over the set of discrete rectangles have been studied in~\cite{siegmund_yakir}, where, in particular, Proposition~\ref{prop:drect} can be found. The arguments of~\cite{siegmund_yakir} are somewhat heuristical; we use a different method.

\section{Asymptotic extreme-value rate}\label{sec:extrme_rate}
It is well-known that the maximum of a large number number of dependent random variables behaves in the same way as the maximum of the same number of independent random variables provided the dependence between the variables is weak enough, see~\cite[Ch.~3,4]{Lead}. It should be stressed that the results of Section~\ref{sec:main} do not fall into this category. To compare the behavior of the maxima of the standardized Gaussian noise to the behavior of independent Gaussian random variables, we introduce a notion of asymptotic extreme-value rate which is of independent interest.
To begin with, recall a well-known fact, see~\cite{Lead}, that if $\{\xi_i,i\in\N\}$ are independent standard Gaussian random variables, then for every $\tau\in\R$,
\begin{equation}\label{eq:equatmain}
\lim_{n\to\infty}\P\left [\max_{i=1,\ldots,n}\xi_i \leq u_n(\tau)\right]=e^{-e^{-\tau}},
\end{equation}
where $u_n(\tau)$ is given by
\begin{equation}\label{eq:def_ab}
u_n(\tau)=\sqrt{2\log n}+\frac{-\frac 12 \log\log n -\log 2\sqrt{\pi}+\tau}{\sqrt{2\log n}},
\;\;\; \tau\in\R.
\end{equation}
\begin{definition}
For every $n\in \N$ let a zero-mean, unit-variance Gaussian field  $\XXX_n=\{\XXX_n(t), t\in T_n\}$ defined on some parameter space $T_n$ be given. Let $f:\N\to\R$ be some function satisfying $\lim_{n\to\infty}f(n)=+\infty$.  We say that the sequence of random fields $\XXX_n$ has extreme-value rate $f$ if for every $\tau\in \R$,
\begin{equation}\label{eq:def_extreme_rate}
\lim_{n\to\infty}\P\left[\sup_{t\in T_n} \XXX_n(t)\leq u_{f(n)}(\tau) \right]=e^{-e^{-\tau}}, 
\end{equation}
where $u_n(\tau)$ is defined as in~\eqref{eq:def_ab}.
\end{definition}
Roughly speaking, condition~\eqref{eq:def_extreme_rate} says that the supremum of $\XXX_n$ has the same asymptotic behavior as the supremum of $f(n)$ i.i.d.\ standard Gaussian variables. The next elementary lemma is useful for computing extreme-value rates.
\begin{lemma}\label{lem:extreme_value}
Let $u_n(\tau)$ be given by~\eqref{eq:def_ab} and let $f(n)=\alpha n^{\beta}(\log n)^{\gamma}$ for some $\alpha,\beta>0$ and $\gamma\in\R$. Then, as $n\to\infty$,
$$
u_{f(n)}(\tau)=\sqrt{2\beta\log n}+\frac{(\gamma-\frac 12) \log(\beta\log n)+\log \frac{\alpha}{2\beta^{\gamma}\sqrt{\pi}}+\tau+o(1)}{\sqrt{2\beta\log n}}.
$$
\end{lemma}
The extreme-value rates of the standardized Gaussian noise over various collections of subsets can be now evaluated by comparing the results of Section~\ref{sec:main} with Lemma~\ref{lem:extreme_value}. 
\begin{center}
\begin{tabular}{|l|l|}
\hline
\rule{0mm}{5mm}Collection of subsets $T_n$& Extreme-value rate $f(n)$\\
\hline
$T_n=\Cube_n$, discrete cubes in $\{1,\ldots,n\}^d$   & $(2d)^{d+1}J_d n^d\log n$\rule{0mm}{5mm}\\
\rule{0mm}{5mm}$T_n=\Rect_n$, discrete rectangles in $\{1,\ldots,n\}^d$ & $(2d)^{2d}G_d^d n^d (\log n)^d$ \rule{0mm}{5mm}\\
$T_n=\CCube_n(a)$, cubes in $[0,n]^d$ with side $\geq a$  & $2E_d\left(\frac{2d}{a}\right)^d n^d (\log n)^{d+1}$\rule{0mm}{5mm}\\
$T_n=\CRect_n(a)$, rectangles in $[0,n]^d$ with sides $\geq a$  & $\frac{d^{2d}}{a^d}n^d (\log n)^{2d}$\rule{0mm}{5mm}\\
\hline
\end{tabular}
\end{center}
\vspace*{0.25cm}

The rest  of the paper is organized as follows.
In Section \ref{sec:locstat} we recall the definition of locally self-similar Gaussian fields. Applications of this notion will be given in Section~\ref{sec:locstat_appl}. In Section~\ref{sec:poisson} we recall a Poisson limit theorem for finite-range dependent events. The proofs of our result are given in Sections~\ref{sec:proof_cont} and~\ref{sec:proof_discr}.

Throughout the paper we use the following notation. We denote $d$-dimensional vectors by $\xx=(x_1,\ldots,x_d)$, $\yy=(y_1,\ldots,y_d)$ etc. We write $\xx\leq \yy$ if $x_i\leq y_i$ for all $1\leq i\leq d$. Given some $\xx,\yy\in\R^d$ with $\xx\leq \yy$, we denote by $[\xx,\yy]$ the $d$-dimensional rectangle $[x_1,y_1]\times\ldots\times [x_d,y_d]$. Given $\xx,\yy\in\Z^d$ with $\xx\leq \yy$, we denote by $[\xx,\yy]_{\Z^d}$ the  discrete $d$-dimensional rectangle consisting of all $\zz\in\Z^d$ such that $x_i\leq z_i<y_i$ for all $1\leq i\leq d$. We denote by $C$ a large positive constant whose value may change from line to line.

\section{Locally self-similar Gaussian fields}\label{sec:locstat}
The main tool of our proofs is the extreme-value theory of continuous-time Gaussian processes initiated by \citet{P1,P2}. Let $\{\XXX(t),t\in \R\}$ be a stationary zero-mean Gaussian process whose covariance function $r(s)=\E[\XXX(0)\XXX(s)]$ satisfies
$r(s)=1-C|s|^{\alpha}+o(|s|^{\alpha})$ as $s\to 0$
for some $\alpha\in(0,2]$, $C>0$, and suppose that $r(s)=1$ holds only for $s=0$.  Under these conditions, \citet{P1} proved the asymptotic equality
\begin{equation}\label{eq:pickands}
\P\left[\sup_{t\in [0,l]} \XXX(t)>u\right]\sim l H_{\alpha}C^{1/\alpha} \frac{1}{\sqrt{2\pi}}u^{2/\alpha-1} e^{-u^2/2}, \qquad u\to+\infty,
\end{equation}
for all $l>0$, where $H_{\alpha}\in (0,\infty)$ is the so-called Pickands constant. Only the values $H_{1}=1$ and $H_2=\pi^{-1/2}$ are known rigorously.
Neither the assumption of stationarity nor the one-dimensionality of the parameter set of the process is relevant for Pickands' approach.
His result has been extended in various directions by~\citet{qualls_watanabe,qualls_watanabe1}, \citet{bickel_rosenblatt}, \citet{huesler}, \citet{albin} (see also the subsequent works by Albin), \citet{piterbarg_fatalov}, \citet{PM}, \citet{ChL}; see also the monographs~\cite{piterbarg}, \cite[Ch.~12]{Lead}. A general non-rigorous approach, called the Poisson clumping heuristics, has been suggested by~\citet{Ald}.

We recall a result of~\citet{ChL}, see also~\cite{PM}, which will play a fundamental role in the sequel. It is an extension of~\eqref{eq:pickands} to the locally self-similar fields (also called locally stationary fields), a class of Gaussian fields satisfying certain local condition.   A function $f:\R^d \to \R$  is called homogeneous of order $\alpha>0$ if $f(\lambda s)=|\lambda|^\alpha f(s)$ for each $s\in \R^d$ and $\lambda\in \R$.
Let $H(\alpha)$ be the linear space of all continuous homogeneous functions of order $\alpha$ endowed with the norm $\|f\|=\sup_{\|t\|_2=1}f(t)$. We may identify $H(\alpha)$
with the  Banach space $C(\mathbb S^{d-1})$
of continuous functions on the unit sphere $\mathbb S^{d-1}$ in $\mathbb R^d$.
Let $H^+(\alpha)$ be the cone of all strictly positive functions in $H(\alpha)$.
\begin{definition}[see \cite{ChL}] \label{def:loc_self_sim}
Let $\{\XXX(t),t\in D\}$ be a zero-mean, unit-variance Gaussian field  defined on some domain  $D\subset \mathbb R^d$. Let $r(t_1,t_2)=\E[\XXX(t_1)\XXX(t_2)]$ be the covariance function of $\XXX$ and suppose that $r(t_1,t_2)<1$ for $t_1\neq t_2$.
The field $\XXX$ is called \textit{locally self-similar} with index  $\alpha\in(0,2]$  if for every $t\in D$ a function $C_t\in H^+(\alpha)$ exists such that the following two conditions hold:
\begin{enumerate}
\item we have
$
\lim_{\|s\|_2 \to 0} \frac {1-r(t,t+s)}{C_t(s)}=1
$
uniformly in $t$ on compact subsets of $D$;
\item the map $C_{\bullet}:D\to H^+(\alpha)$, sending $t$ to $C_t$, is continuous.
\end{enumerate}
\end{definition}
The collection of homogeneous functions $C_{t}$, $t\in D$, is referred to as the \textit{local structure} of the field~$\XXX$. It can be shown~\cite{kabluchko_stand_gauss} that for every $t\in D$ there exists a finite measure $\Gamma_t$ on $\mathbb S^{d-1}$ such that the following representation holds
$$
C_t(s)=\int_{\mathbb S^{d-1}}|\langle s,x\rangle|^{\alpha} d\Gamma_t(x).
$$

The next theorem, proved in~\cite{ChL} (see also~\cite{PM} for a similar result), describes the asymptotic behavior of the high excursion probability of a locally self-similar Gaussian field.
\begin{theorem}[see~\cite{ChL, PM}]\label{theo:locstat}
Let $\{\XXX(t),t\in D\}$ be a  Gaussian field defined on some domain  $D\subset \mathbb R^d$. 
Suppose that $\XXX$ is locally self-similar of index $\alpha$ with local structure $C_t(s)$. Let $K\subset D$ be a
compact set with positive Jordan measure. Then, as $u\to\infty$,
\begin{equation}\label{eq:chan_lai}
\P\left[\sup_{t\in K} \XXX(t)>u \right]\sim  \frac 1 {\sqrt {2\pi}}\left(\int_K \Lambda(t)dt\right)  \;u^{\frac {2d}{\alpha}-1}e^{-u^2/2},
\end{equation}
where the function $\Lambda:D\to(0,\infty)$ is defined in~\eqref{eq:def_high_cross_int} below.
\end{theorem}
The function $\Lambda$, which might be called the \textit{high excursion intensity} of $\XXX$, is defined as follows.
For each $t\in D$, let $\{Y_t(s),s\in \R^d \}$ be a Gaussian field such that for all $s,s_1,s_2\in \R^d$,
\begin{align}\label{eq:def_proc_Y}
\E Y_t(s)=-C_t(s),\;
\Cov (Y_t(s_1),Y_t(s_2))=C_t(s_1)+C_t(s_2)-C_t(s_1-s_2).
\end{align}
The field $Y_t$ describes the local behavior of the field $\XXX$ conditioned to reach an  extremely high value at $t$ and will be therefore called the \textit{tangent process} of $\XXX$ in the sequel. Note that $\{Y_t(s)+C_t(s), s\in\R^d\}$ is a zero-mean self-similar Gaussian field with stationary increments, a non-isotropic generalization of the fractional Brownian motion.
With the above notation, it has been shown in~\cite{ChL} that the following limit exists in $(0,\infty)$ and is a continuous function of  $t$:
\begin{equation}\label{eq:def_high_cross_int}
\Lambda(t)=\lim_{T\to\infty} \frac{1}{T^d}\E\left[\exp
\left(\sup_{s\in[0,T]^d}Y_t(s)\right) \right].
\end{equation}


The following theorem has been obtained as a by-product of Theorem~\ref{theo:locstat} in~\cite{ChL}. It describes the asymptotic behavior of the high excursion probability over a finite grid with mesh size going to $0$. For one-dimensional stationary processes it can be found in~\cite[Lemma 12.2.4]{Lead}.
\begin{theorem}\label{theo:locstatdiscr}
Suppose that  the conditions of Theorem~\ref{theo:locstat} are satisfied.
Let  $u\uparrow \infty$ and $q\downarrow 0$ in such a way that $q u^{2/\alpha}\to \ka$ for some constant $\ka>0$. Then,
\begin{equation}\label{eq:chan_lai_discr}
\P\left[\max_{t\in K\cap q\Z^d} \XXX(t)>u \right]\sim  \frac 1 {\sqrt {2\pi}}\left(\int_K \Lambda(t; \ka)dt\right)  \;u^{\frac {2d}{\alpha}-1}e^{-u^2/2},
\end{equation}
where
\begin{equation}\label{eq:def_Ha}
\Lambda(t; \ka)=\lim_{T\to\infty} \frac{1}{T^d}\E\left[\exp\left(\sup_{s\in[0,T]^d\cap \ka\Z^d}Y_t(s)\right) \right].
\end{equation}
Furthermore, $\lim_{\ka\downarrow 0}\Lambda(t; \ka)=\Lambda(t)$, where $\Lambda(t)$ is the high excursion intensity of $\XXX$.
\end{theorem}

\section{Applications to the standardized Gaussian noise}\label{sec:locstat_appl}
The next example, see~\cite[J25]{Ald}, \cite{ChL}, will play a major role in the sequel.
Let $\{B(t), t\in\R\}$ be the standard Brownian motion. Let $\DInt=\{(x,y)\in\R^2 | x<y\}$ be the space of all intervals in $\R$. Define a Gaussian field $\{\U(x,y), (x,y)\in \DInt\}$ by
\begin{equation}\label{eq:def_UUU}
\U(x,y)=\frac{B(y)-B(x)}{\sqrt{y-x}}.
\end{equation}
It is elementary to verify, see~\cite[J25]{Ald}, \cite{ChL}, \cite{kabluchko_stand_gauss}, that the field $\U$ is locally self-similar with index $\alpha=1$ and a  local structure given by
\begin{equation}\label{eq:loc_struct_onedim}
C_{(x,y)}(p,q)=\frac 12 \cdot \frac{|p|+|q|}{y-x}, \;\;\;
(x,y)\in \DInt, \;\;\;(p,q)\in \R^2.
\end{equation}

We will be interested in multidimensional generalizations of this example. Recall that a rectangle in $\R^d$ is a set of the form $[\xx,\yy]=[x_1,y_1]\times\ldots\times [x_d,y_d]$. We can identify the collection $\CRect$ of rectangles with $\DInt^d$. Also, we will sometimes identify a rectangle $[\xx,\yy]\in\CRect$ with a pair $(\xx,\yy)\in\R^{2d}$.
Denote by $\{\W(A), A\in \CRect\}$ a white noise indexed by rectangles. Recall our standing assumption that the sample paths of $\W$ are continuous.
Define a Gaussian random field $\{\XXX(A), A\in \CRect\}$, called the standardized white noise, by
\begin{equation}\label{eq:def_XXX}
\XXX(A)=\frac{\W(A)}{\sqrt{|A|}}.
\end{equation}
The covariance function $r_{\XXX}$ of the random field $\{\XXX(A), A\in \CRect\}$ is an $n$-fold tensor product of the covariance function $r_{\U}$ of the process $\U$ from~\eqref{eq:def_UUU}, i.e.,
if $A=[\xx,\yy]$ and $A'=[\xx',\yy']$ are in $\CRect$, then
\begin{equation}\label{eq:cov_XXX_prod_form}
r_{\XXX}(A,A')=\prod_{i=1}^d r_{\U}((x_i,y_i),(x_i',y_i')).
\end{equation}
\begin{proposition}\label{theo:tail_crect}
Let $\{\XXX(A),A\in\CRect\}$ be the standardized white noise defined in~\eqref{eq:def_XXX}. Let $\RR\subset \CRect$ be a compact subset of the space of rectangles $\CRect$ having a  positive Jordan measure. Then, as $u\to+\infty$,
\begin{equation}
\P\left [\sup_{A\in \RR} \XXX(A) > u \right]\sim  \frac{1}{4^{d}\sqrt{2\pi}} \left(\int_{\RR} \frac {d\xx d\yy}{\prod_{i=1}^d(y_i-x_i)^2}\right)  u^{4d-1} e^{- u^2/2}.
\end{equation}
\end{proposition}
\begin{proof}
It follows from Definition~\ref{def:loc_self_sim}, Eqn.~\eqref{eq:loc_struct_onedim} and the tensor product structure~\eqref{eq:cov_XXX_prod_form} that the random field $\{\XXX(A), A\in \CRect\}$ is locally self-similar with index $\alpha=1$ and its local structure is given by
\begin{equation}\label{eq:local_rect}
C_{\xx,\yy}(\pp,\qq)=\frac 12\sum_{i=1}^d \frac{|p_{i}|+|q_{i}|}{y_i-x_i},\;\;\; \pp,\qq\in\R^d.
\end{equation}
For $1\leq i\leq d$ let $\{V_i(s),s\in\R\}$ and $\{W_i(s), s\in\R\}$  be independent standard Brownian motions with drift $-|s|/2$.  It follows from~\eqref{eq:def_proc_Y} and~\eqref{eq:local_rect} that the tangent process of $\{\XXX(A), A\in\CRect\}$ at $[\xx,\yy]\in\CRect$ is given  (in distribution) by
\begin{equation}\label{eq:def_tang_rect}
Y_{\xx,\yy}(\pp,\qq)=
\sum_{i=1}^d \left(V_i\left(\frac{p_i}{y_i-x_i}\right)+W_i\left(\frac{q_i}{y_i-x_i}\right)\right),
\;\;\;\pp,\qq\in\R^d.
\end{equation}
We compute the high excursion intensity $\Lambda$ given in~\eqref{eq:def_high_cross_int}. Using a simple change of variables and the independence of the processes $V_1,W_1,\ldots,V_d,W_d$, we obtain that
\begin{align*}
\Lambda(\xx,\yy)
&=
\frac{1}{4^d\prod_{i=1}^d(y_i-x_i)^2}
\lim_{T\to\infty}\frac{1}{T^{2d}}\E\left[\exp\sup_{\pp,\qq\in[0,T]^{d}}\sum_{i=1}^d(V_i(2p_i)+W_i(2q_i)) \right]\\
&=
\frac{1}{4^d\prod_{i=1}^d(y_i-x_i)^2}
\left\{\lim_{T\to\infty}\frac{1}{T}
\E\left[\exp\sup_{p_1\in[0,T]}V_1(2p_1) \right]\right\}^{2d}.
\end{align*}
The limit on the right-hand side is the Pickands constant $H_1=1$; see~\cite{P1}. Hence, $\Lambda(\xx,\yy)=4^{-d}\prod_{i=1}^d(y_i-x_i)^{-2}$. The proposition follows now by  Theorem~\ref{theo:locstat}.
\end{proof}

Next we prove a similar result for the maximum of the standardized white noise over a subset of the space of cubes. We identify a cube $[\xx,\xx+h]$ with the point $(\xx,h)\in \R^d\times (0,\infty)$. 
\begin{proposition}\label{theo:tail_ccube}
Let $\{\XXX(A), A\in\CRect\}$ be the standardized white noise as in~\eqref{eq:def_XXX}. Let $\RR\subset \CCube$ be a compact subset of the space of cubes $\CCube$ having a positive Jordan measure. There is a constant $E_d>0$ such that as $u\to+\infty$,
\begin{equation}
\P\left [\sup_{A\in \RR} \XXX(A) > u \right ]\sim  \frac{E_d}{\sqrt{2\pi}} \cdot\left(\int_{\RR}\frac {d\xx dh}{h^{d+1}}\right) u^{2d+1} e^{- u^2/2}.
\end{equation}
\end{proposition}
\begin{proof}
The random field $\{\XXX(A), A\in \CCube\}$ is locally self-similar with index $\alpha=1$ since it is a restriction of a locally self-similar field $\{\XXX(A), A\in \CRect\}$ to a linear subspace. Let $V_i$ and $W_i$ be drifted Brownian motions as in the previous proof. The tangent process defined in~\eqref{eq:def_proc_Y} is given by
\begin{equation}\label{eq:def_tang_cube}
Y_{\xx,h}(\pp,g)=\sum_{i=1}^d \left(V_i\left(\frac{p_i}{h}\right)+W_i\left(\frac{p_i+g}{h}\right)\right),\;\;\; \pp\in\R^d, g\in\R.
\end{equation}
The  high excursion intensity $\Lambda(\xx,h)$ defined in~\eqref{eq:def_high_cross_int} is given by
\begin{align*}
\Lambda(\xx,h)
=
\lim_{T\to\infty} \frac{1}{T^{d+1}}\E\left[\exp
\sup_{\substack{\pp\in[0,T]^d\\g\in[0,T]}}\sum_{i=1}^d \left(V_i\left(\frac{p_i}{h}\right)+W_i\left(\frac{p_i+g}{h}\right)\right)\right].
\end{align*}
A change of variables shows that we have $\Lambda(\xx,h)=h^{-(d+1)}E_d$ where $E_d$ is a constant given by
\begin{equation}\label{eq:def_Ed}
E_d=\lim_{T\to\infty} \frac{1}{T^{d+1}}\E\left[\exp
\sup_{\substack{\pp\in[0,T]^d\\g\in[0,T]}}\sum_{i=1}^d \left(V_i\left(p_i\right)+W_i\left(p_i+g\right)\right)\right].
\end{equation}
The proof is completed by applying Theorem~\ref{theo:locstat}.
\end{proof}
Next we consider the maximum of the standardized Gaussian noise taken over a set of discrete rectangles. Let $\{B(s), s\geq 0\}$ be a standard Brownian motion. For $h,\ka>0$ define
\begin{equation}\label{eq:def_G}
G(h;\ka)=\frac 1{h^2} F^2\left(\frac{\ka}{h}\right),
\;\;\;
F(\ka)=\lim_{T\to\infty} \frac 1 T \E \left[\exp \sup_{s\in [0,T]\cap \ka\Z} \left(B(s)-\frac s2\right)\right].
\end{equation}
Using the fluctuation theory of random walks it can be shown, see~\cite{kabluchko_stand_gauss}, that with $\bar \Phi(u)=\frac{1}{\sqrt{2\pi}}\int_u^{\infty}e^{-w^2/2}dw$ denoting the tail  of the standard normal law,
\begin{equation}\label{eq:def_Fa1}
F(\ka)=\frac 1 {\ka} \exp\left\{-2\sum_{n=1}^{\infty}\frac 1 n \bar \Phi\left(\frac 12 \sqrt {\ka n}\right)\right\}.
\end{equation}

\begin{proposition}\label{prop:drect}
Let $\{\XXX(A),A\in\CRect\}$ be the standardized white noise as in~\eqref{eq:def_XXX}. Let $\RR\subset  \CRect$ be a compact subset of the space of rectangles $\CRect$ having a positive Jordan measure. Let  $u\uparrow\infty$ and $q\downarrow 0$ in such a way that $q u^2\to \ka$ for some constant $\ka>0$. Then,
\begin{equation}
\P\left[\sup_{A\in \RR \cap q\Z^{2d}} \XXX(A)>u \right]
\sim
\frac 1 {\sqrt{2\pi}}\left(\int_{\RR} \prod_{i=1}^dG(y_i-x_i;\ka) d\xx d\yy \right) u^{4d-1}e^{-u^2/2}.
\end{equation}
\end{proposition}
\begin{proof}
We compute the function $\Lambda(\xx,\yy;\ka)$ given in Theorem~\ref{theo:locstatdiscr}. Recall that the tangent process is given by~\eqref{eq:def_tang_rect}. Setting $h_i:=y_i-x_i$ and making a change of variables, we obtain that $\Lambda(\xx,\yy;\ka)$ is equal to
\begin{align*}
\frac{1}{\prod_{i=1}^dh_i^2} \lim_{T\to\infty} \frac{1}{T^{2d}}\E \exp \left( \sum_{i=1}^d \sup_{\substack{p_i\in[0,T]\\p_i\in \ka h_i^{-1}\Z}}V_i(p_i)+\sum_{i=1}^d \sup_{\substack{q_i\in[0,T]\\q_i\in \ka h_i^{-1}\Z}} W_i(q_i)\right).
\end{align*}
Consequently, by the independence of $V_1,\ldots,V_d, W_1,\ldots,W_d$ and~\eqref{eq:def_G}, we obtain $\Lambda(\xx,\yy;\ka)=G(h_1;\ka)\ldots G(h_d;\ka)$. The proposition follows from Theorem~\ref{theo:locstatdiscr}.
\end{proof}

\begin{proposition}\label{prop:dcube}
Let $\{\XXX(A),A\in\CRect\}$ be the standardized white noise defined in~\eqref{eq:def_XXX}. Let $\RR\subset \CCube$ be a compact subset of the space of cubes $\CCube$ having a positive Jordan measure. Let  $u\uparrow\infty$ and $q\downarrow 0$ in such a way that $q u^2\to \ka$ for some constant $\ka>0$. Then, with a function $J_d(h;\ka)$ defined in~\eqref{eq:def_Jd_a} and~\eqref{eq:def_Ed_a} below,
\begin{equation}
\P\left[\sup_{A\in \RR \cap q\Z^{d+1}} \XXX(A)>u \right]
\sim
\frac 1 {\sqrt{2\pi}}\left(\int_{\RR}J_d(h; \ka) d\xx dh \right) u^{2d+1}e^{-u^2/2}.
\end{equation}
\end{proposition}
\begin{proof}
Recall that the tangent process is given by~\eqref{eq:def_tang_cube}. The high excursion intensity $\Lambda(\xx,h;\ka)$ defined in Theorem~\ref{theo:locstatdiscr} is given by
\begin{align*}
\Lambda(\xx,h;\ka)
=
\lim_{T\to\infty} \frac{1}{T^{d+1}}\E\left[\exp
\sup_{\substack{\pp\in[0,T]^d\cap \ka\Z^d\\g\in[0,T]\cap \ka\Z}}\sum_{i=1}^d \left(V_i\left(\frac{p_i}{h}\right)+W_i\left(\frac{p_i+g}{h}\right)\right)\right].
\end{align*}
By a change of variables,  we have $\Lambda(\xx,h;\ka)=J_d(h;\ka)$, where
\begin{align}
J_d(h;\ka)&=h^{-(d+1)}E_d(\ka/h),\label{eq:def_Jd_a}\\
E_d(\ka)&=\lim_{T\to\infty} \frac{1}{T^{d+1}}\E\left[\exp
\sup_{\substack{\pp\in[0,T]^d\cap \ka\Z^d\\g\in[0,T]\cap \ka\Z}}\sum_{i=1}^d \left(V_i\left(p_i\right)+W_i\left(p_i+g\right)\right)\right].\label{eq:def_Ed_a}
\end{align}
The proposition follows from Theorem~\ref{theo:locstatdiscr}.
\end{proof}
\begin{lemma}\label{lem:JG_finite}
Define $J_d(h)=J_d(h;2d)$ by~\eqref{eq:def_Jd_a},\eqref{eq:def_Ed_a} and $G_d(h)=G(h;2d)$ by~\eqref{eq:def_G} with $\kappa=2d$. Then, $J_d:=\int_{0}^{\infty}J_d(h)dh<\infty$ and $G_d:=\int_0^{\infty}G_d(h)dh<\infty$.
\end{lemma}
\begin{proof}
We have $\lim_{\ka\downarrow 0}F(\kappa)=1/2$ by~\cite[Ch.~12]{Lead}  and $\lim_{\ka\downarrow 0}E_d(\kappa)=E_d$ by~\cite{ChL}. It follows that $G_d(h)\sim 1/(4h^2)$ and $J_d(h)\sim E_d/h^{d+1}$ as $h\uparrow +\infty$. Hence, $G_d<\infty$ and $J_d<\infty$.
\end{proof}

\section{A Poisson limit theorem for dependent events}\label{sec:poisson}
In this section we recall a Poisson limit theorem for finite-range dependent random  events which will be needed in our proofs. A more general statement can be found in~\cite[Thm.~1]{arratia_etal}.  For every $N\in\N$ let $E_{1N},\ldots, E_{NN}$ be (in general, dependent) random events such that $\P[E_{iN}]=p_N$ for all $1\leq i\leq N$, where $p_N$ is a sequence satisfying $\lambda:=\lim_{N\to\infty} Np_N\in (0,\infty)$. 
We assume that the events $E_{1N},\ldots, E_{NN}$ are finite-range dependent in the following sense: there exist $B_{1N},\ldots,B_{NN}$, subsets of $\{1,\ldots,N\}$, such that the following three conditions are satisfied:
\begin{enumerate}
\item \label{cond:poi1} There is a constant $C>0$ not depending on $N$ such that $|B_{iN}|<C$ for every  $1\leq i\leq N$ and $N\in\N$.
\item \label{cond:poi2} For every $1\leq i\leq N$, the random event $E_{iN}$ is independent of the collection  $\{E_{jN}, j\notin B_{iN}\}$.
\item \label{cond:poi3} We have $\lim_{N\to\infty}\sum_{i=1}^N\sum_{j\in B_{iN}\backslash \{i\}} \P[E_{iN}\cap E_{jN}]=0.$
\end{enumerate}
\begin{theorem}\label{theo:poisson}
Under the above assumptions, the distribution of the  random variable $\sum_{i=1}^N \boldsymbol{1}_{E_{iN}}$ converges as $N\to\infty$ to the Poisson distribution with mean $\lambda$. In particular,
\begin{equation}
\lim_{N\to\infty}\P\left[\cap_{i=1}^N E_{iN}^c\right]=e^{-\lambda}.
\end{equation}
\end{theorem}

\section{Proofs in the continuous-time case}\label{sec:proof_cont}
\subsection{Proof of Theorem~\ref{theo:main_ccube}}\label{sec:proof_ccube}
Let $\XXX(A)=\W(A)/\sqrt{|A|}$ be the standardized white noise. Recall that $\CCube$ is the set of all $d$-dimensional cubes and $\CCube_n$ is the set of all $d$-dimensional cubes contained in $[0,n]^d$. For $0\leq a<b\leq n$ let $\CCube_n(a,b)$ be the set of all cubes $[\xx,\xx+h]\in \CCube_n$ such that  $h\in [a,b]$.
\begin{lemma}\label{lem:ccube1}
Fix some $\tau\in\R$ and let $u_n=u_n(\tau)$ be defined by~\eqref{eq:def_ab_ccube}. We have
$$
\lim_{n\to\infty}\P\left[\sup_{A\in \CCube_n(a,b)} \XXX(A)\leq u_n \right]=e^{-e^{-\tau}\left(1-\left(\frac {a}{b}\right)^d\right)}.
$$
\end{lemma}
\begin{proof}
For $\kk\in\Z^d$ define $\RR_{\kk}$ to be the set of cubes of the form $[\xx,\xx+h]\in\CCube$ such that $h\in [a,b]$ and $x_i\in [k_i, k_i+1]$ for every $1\leq i\leq d$. Let $E_{\kk,n}$ be the random event $\{\sup_{A\in \RR_{\kk}} \XXX(A)>u_n\}$.
Note that by the translation invariance, the probability $p_n:=\P[E_{\kk,n}]$ is independent of $\kk\in\Z^d$. By Proposition~\ref{theo:tail_ccube} and~\eqref{eq:def_ab_ccube}, we have as $n\to\infty$,
\begin{equation}\label{eq:eq1}
p_n
\sim
\frac{E_d}{\sqrt{2\pi}}u_n^{2d+1} e^{-u_n^2/2} \int_{\RR_{\kk}}\frac{d\xx dh}{h^{d+1}}
\sim
\frac {da^de^{-\tau}} {n^d} \int_{a}^b \frac{dh}{h^{d+1}}
\sim
\frac {e^{-\tau}} {n^d}\left(1-\left(\frac ab\right)^d\right).
\end{equation}
To prove the lemma, we will verify the assumptions of Section~\ref{sec:poisson}. Define index sets  $K'_n=[0,n]^d\cap \Z^d$ and $K''_n=[0,n-b-1]^d\cap \Z^d$. Clearly, we have
\begin{equation}\label{eq:ccube_inclusion}
\bigcap_{\kk\in K'_n} E_{\kk,n}^c\subset \left\{\sup_{A\in \CCube_n(a,b)} \XXX(A)\leq u_n\right\}\subset \bigcap_{\kk\in K''_n} E_{\kk,n}^c.
\end{equation}
We will show that the families of random events $\{E_{\kk,n}, \kk\in K'_n\}$ and $\{E_{\kk,n}, \kk\in K''_n\}$ satisfy the assumptions of Theorem~\ref{theo:poisson}. Note that we have $|K_n'|\sim n^d$ and $|K_n''|\sim n^d$ as $n\to\infty$. Hence, by~\eqref{eq:eq1},
$$
\lim_{n\to\infty}p_n|K_n'|=\lim_{n\to\infty}p_n|K_n''|=e^{-\tau}\left(1-\left(\frac ab\right)^d\right).
$$
The events $E_{\kk,n}$ are finite-range dependent, i.e., $E_{\kk_1,n}$ and $E_{\kk_2,n}$ are independent provided that $\|\kk_1-\kk_2\|_{\infty}>b+1$. Hence, conditions~\ref{cond:poi1} and~\ref{cond:poi2} of Section~\ref{sec:poisson} are satisfied. By Proposition~\ref{theo:tail_ccube}, 
\begin{equation}\label{eq:eq3}
\P[E_{\kk_1,n}\cup E_{\kk_2,n}]\sim
\frac{E_d}{\sqrt{2\pi}}u_n^{2d+1} e^{-u_n^2/2} \int_{\RR_{\kk_1}\cup\RR_{\kk_2}}\frac{d\xx dh}{h^{d+1}}
\sim
2\frac {e^{-\tau}} {n^d} \left(1-\left(\frac ab\right)^d\right).
\end{equation}
It follows from~\eqref{eq:eq1} and~\eqref{eq:eq3} that uniformly in $\kk_1,\kk_2\in \Z^d$, we have $\P[E_{\kk_1,n}\cap E_{\kk_2,n}]=o(1/n^{d})$ as $n\to\infty$. Together with the finite-range dependence, this implies that condition~\ref{cond:poi3} of Section~\ref{sec:poisson} is satisfied.  By Theorem~\ref{theo:poisson}, we have
\begin{equation}\label{eq:lem35}
\lim_{n\to\infty}\P\left[\cap_{\kk\in K'_n} E_{\kk,n}^c\right]
=
\lim_{n\to\infty}\P\left[\cap_{\kk\in K''_n} E_{\kk,n}^c\right]
=
e^{-e^{-\tau}\left(1-\left(\frac ab\right)^d\right)}.
\end{equation}
The statement of the lemma follows from~\eqref{eq:ccube_inclusion} and~\eqref{eq:lem35}.
\end{proof}

To prove Theorem~\ref{theo:main_ccube} we need to take the limit $b\to+\infty$ in Lemma~\ref{lem:ccube1}. The next lemma estimates the high excursion probability over $\CCube_n(b)=\CCube_n(b,n)$, the set of all cubes $[\xx,\xx+h]\in\CCube_n$ such that $h\geq b$.
\begin{lemma}\label{lem:ccube2}
There is a constant $C$ such that for every $0<b<n$ and $u>1$,
$$
\P\left[\sup_{A\in \CCube_n(b,n)} \XXX(A)>u \right]
\leq
Cb^{-d} u^{2d+1}e^{-u^2/2}n^d.
$$
\end{lemma}
\begin{proof}
For $\kk\in\Z^d$ and $l\in\Z$ denote by $\RR_{\kk,l}$ the set of all cubes $[\xx,\xx+h]\in \CCube$ such that $h\in [2^l,2^{l+1}]$ and $x_i\in [2^{l}k_i, 2^l(k_i+1)]$ for all $1\leq i\leq d$.
By  Proposition~\ref{theo:tail_ccube}, for every $u>1$,
\begin{equation}\label{eq:eq4}
\P\left[\sup_{A\in \RR_{\kk,l}} \XXX(A)>u\right ]
\leq
C u^{2d+1}e^{-u^2/2} \int_{\RR_{\kk,l}}\frac{d\xx dh}{h^{d+1}}
\leq
Cu^{2d+1}e^{-u^2/2}.
\end{equation}
Note that the constant $C$ is independent of $\kk,l$ since the left-hand side does not depend on $\kk,l$ by the affine invariance. Without restriction of generality, we may assume that $n=2^{L}$ and $b=2^{L_b}$ for some $L,L_b\in\Z$. Otherwise, we may replace $n$ by $2^{\lceil\log_2 n\rceil}$ and $b$ by $2^{\lfloor \log_2 b\rfloor}$.
The set $\CCube_n(b,n)$ can be written as a union of sets of the form $\CCube_n(2^l,2^{l+1})$, $l=L_b,\ldots, L-1$. Now, the set $\CCube_n(2^l,2^{l+1})$ can be covered by $n^d/2^{ld}$ sets of the form $\RR_{\kk,l}$. Hence, we can cover the set $\CCube_n(b,n)$ by at most $\sum_{l=L_b}^{L-1} (n^d/2^{ld})\leq 2n^d/b^d$ sets of the form $\RR_{\kk,l}$. The statement of the lemma follows by applying to each of these sets~\eqref{eq:eq4}.
\end{proof}
We are now in position to complete the proof of Theorem~\ref{theo:main_ccube}. Let $u_n=u_n(\tau)$ be chosen as in~\eqref{eq:def_ab_ccube}. We have $\CCube_n(a,b)\subset \CCube_n(a,n)$ for every $0\leq a<b\leq n$. By Lemma~\ref{lem:ccube1} we obtain
\begin{align}
\limsup_{n\to\infty}\P\left[\sup_{A\in \CCube_n(a,n)} \XXX(A)\leq u_n \right]
&\leq
\lim_{n\to\infty}\P\left[\sup_{A\in \CCube_n(a,b)} \XXX(A)\leq u_n \right]\label{eq:eq5}\\
&=
e^{-e^{-\tau}\left(1-\left(\frac ab\right)^d \right)}.\notag
\end{align}
Let us prove a converse inequality. By~\eqref{eq:def_ab_ccube}, we have $u_n^{2d+1}e^{-u_n^2/2}\leq Cn^{-d}$. Note that $\CCube_n(a,n)\backslash \CCube_n(a,b)=\CCube_n(b,n)$. It follows from  Lemma~\ref{lem:ccube2} that
$$
\P\left[\sup_{A\in \CCube_n(a,n)\backslash \CCube_n(a,b)} \XXX(A)>u_n\right]
\leq
\frac C{b^d}.
$$
Consequently, by Lemma~\ref{lem:ccube1} we have
\begin{align}
\liminf_{n\to\infty}\P\left[\sup_{A\in \CCube_n(a,n)} \XXX(A)\leq u_n \right]
&\geq \lim_{n\to\infty}\P\left[\sup_{A\in \CCube_n(a,b)} \XXX(A)\leq u_n \right]-\frac C{b^d} \label{eq:eq6}\\
&= e^{-e^{-\tau}\left(1-\left(\frac ab\right)^d\right)}-\frac C{b^d}.\notag
\end{align}
The proof is completed by letting $b\to+\infty$ in~\eqref{eq:eq5} and~\eqref{eq:eq6}.

\subsection{Proof of Theorem~\ref{theo:main_crect}}
Let $\XXX(A)=\W(A)/\sqrt{|A|}$ be the standardized white noise. Recall that $\CRect$ is the set of all $d$-dimensional rectangles and $\CRect_n$ is the set of all $d$-dimensional rectangles contained in $[0,n]^d$. For $0\leq a<b\leq n$ let $\CRect_n(a,b)$ be the set of all rectangles $[\xx,\xx+\hh]\in \CRect_n$ such that $h_i\in [a,b]$ for all $1\leq i\leq d$.
\begin{lemma}\label{lem:crect1}
Fix some $\tau\in\R$ and let $u_n=u_n(\tau)$ be defined by~\eqref{eq:def_ab_crect}. Then,
$$
\lim_{n\to\infty}\P\left[\sup_{A\in \CRect_n(a,b)} \XXX(A)\leq u_n\right]=
e^{-e^{-\tau}\left(1-\frac{a}{b}\right)^d}.
$$
\end{lemma}
\begin{proof}
For $\kk\in\Z^d$ let $\RR_{\kk}$ be the set of all rectangles of the form $[\xx,\xx+\hh]\in\CRect$, such that $x_i\in[k_i, k_i+1]$ and $h_i\in [a,b]$ for every $1\leq i\leq d$. Let $E_{\kk,n}$  be the random event $\{\sup_{A\in \RR_{\kk}} \XXX(A)>u_n\}$. The probability $p_n:=\P[E_{\kk,n}]$ does not depend on $\kk$ by translation invariance. By Proposition~\ref{theo:tail_crect} and~\eqref{eq:def_ab_crect},
\begin{equation}
p_n
\sim
\frac{1}{4^{d}\sqrt{2\pi}}u_n^{4d-1}e^{-u_n^2/2} \int_{\RR_{\kk}}\frac{d\xx d\yy}{\prod_{i=1}^d (y_i-x_i)^2}
\sim
\frac {e^{-\tau}} {n^d} \left(1-\frac {a}{b}\right)^d , \;\;\;n\to\infty.
\end{equation}
The same argument as in the proof of Lemma~\ref{lem:ccube1} shows that the conditions of Section~\ref{sec:poisson} are satisfied. The statement of the lemma follows from Theorem~\ref{theo:poisson} applied to the events $E_{\kk,n}$.
\end{proof}


In the next lemma we will estimate the high excursion probability over the set $\CRectm_n(a,b):=\CRect_n(a,n)\backslash\CRect_n(a,b)$.
\begin{lemma}\label{lem:crect2}
There is a constant $C$ such that for every $0<a\leq b<n$ and $u>1$,
$$
\P\left[\sup_{A\in \CRectm_n(a,b)} \XXX(A)>u \right]\leq C b^{-1}a^{-(d-1)}u^{4d-1}e^{-u^2/2}n^d.
$$
\end{lemma}
\begin{proof}
We may write $\CRectm_n(a,b)=\cup_{m=1}^d \CRectm_{n,m}(a,b)$, where $\CRectm_{n,m}(a,b)$, $1\leq m\leq d$, is the set of all rectangles $[\xx,\xx+\hh]\in \CRect_n$ such that that $h_m\geq b$ and $h_i\geq a$ for all $1\leq i\leq d$.
It suffices to estimate the high-excursion probability over the set $\CRectm_{n,1}(a,b)$. For $\kk\in\Z^d$ and $\lbold\in\Z^d$ denote by $\RR_{\kk,\lbold}$ the set of all rectangles $[\xx,\xx+\hh]\in \CRect$ such that $x_i\in [2^{l_i}k_i, 2^{l_i}(k_i+1)]$ and $h_i\in [2^{l_i},2^{l_i+1}]$ for all $1\leq i\leq d$.
By  Proposition~\ref{theo:tail_ccube}, we have for every $u>1$,
$$
\P\left[\sup_{A\in \RR_{\kk,\lbold}} \XXX(A)>u\right ]\leq  C u^{4d-1}e^{-u^2/2} \int_{\RR_{\kk,\lbold}}\frac{d\xx d\yy}{\prod_{i=1}^d (y_i-x_i)^2}<Cu^{4d-1}e^{-u^2/2}.
$$
The constant $C$ does not depend on $\kk,\lbold$ since the left-hand side does not depend on $\kk,\lbold$ by the affine invariance.
To complete the proof we will show that the set $\CRectm_{n,1}(a,b)$ can be covered by at most $Cb^{-1}a^{-(d-1)}n^d$ sets of the form $\RR_{\kk,\lbold}$. Without restriction of generality we may assume that $n=2^L$, $a=2^{L_a}$, $b=2^{L_b}$ for some $L,L_a,L_b\in\Z$. For $\lbold\in\Z^d$ denote by $\CRectmm_n(\lbold)$ the set of all rectangles $[\xx,\xx+\hh]\in\CRect_n$ such that $h_i\in [2^{l_i},2^{l_i+1}]$ for all $1\leq i\leq d$. Clearly, for every fixed $\lbold\in\Z^d$ the set $\CRectmm_n(\lbold)$ can be covered by $n^d/\prod_{i=1}^d 2^{l_i}$ sets of the form $\RR_{\kk,\lbold}$ with varying $\kk$. We have
$$
\CRectm_{n,1}(a,b)= \cup_{l_1=L_b}^{L-1}\cup_{l_2=L_a}^{L-1}\ldots\cup_{l_d=L_a}^{L-1}  \CRectmm_n(\lbold).
$$
Hence, the set $\CRectm_{n,1}(a,b)$ can be covered by at most
$$
n^d\sum_{l_1=L_b}^{L-1}\sum_{l_2=L_a}^{L-1}\ldots\sum_{l_d=L_a}^{L-1} \prod_{i=1}^d 2^{-l_i}
\leq
n^d\left(\sum_{l=L_b}^{\infty}2^{-l}\right)\left(\sum_{l=L_a}^{\infty} 2^{-l}\right)^{d-1}
\leq Cb^{-1}a^{-(d-1)}n^d
$$
sets of the form $\RR_{\kk,\lbold}$.
This completes the proof of the lemma.
\end{proof}
The proof of Theorem~\ref{theo:main_crect} can be completed as follows.
Choose $u_n=u_n(\tau)$  as in~\eqref{eq:def_ab_crect}. Since $\CRect_n(a,b)\subset\CRect_n(a,n)$ for every $0\leq a<b\leq n$, it follows from Lemma~\ref{lem:crect1} that
\begin{equation}\label{eq:eq7}
\limsup_{n\to\infty}\P\left[\sup_{A\in \CRect_n(a,n)} \XXX(A)\leq u_n \right]
\leq
e^{-e^{-\tau}\left(1-\frac{a}{b}\right)^d}.
\end{equation}
The converse inequality can be proved as follows. Fix some $a>0$. By~\eqref{eq:def_ab_crect}, we have $u_n^{4d-1}e^{-u_n^2/2}\leq Cn^{-d}$. It follows from  Lemma~\ref{lem:crect2} that
$$
\P\left[\sup_{A\in \CRect_n(a,n)\backslash \CRect_n(a,b)} \XXX(A)>u_n\right]
\leq
Cb^{-1}.
$$
Consequently, by Lemma~\ref{lem:crect1} we have
\begin{align}
\liminf_{n\to\infty}\P\left[\sup_{A\in \CRect_n(a,n)} \XXX(A)\leq u_n \right]
&\geq \lim_{n\to\infty}\P\left[\sup_{A\in \CRect_n(a,b)} \XXX(A)\leq u_n \right]- Cb^{-1} \label{eq:eq8}\\
&= e^{-e^{-\tau}\left(1-\frac {a}{b}\right)^d}-Cb^{-1}.\notag
\end{align}
The proof of Theorem~\ref{theo:main_crect} is completed by letting $b\to+\infty$ in~\eqref{eq:eq7} and~\eqref{eq:eq8}.

\section{Proofs in the discrete-time case}\label{sec:proof_discr}
\subsection{Proof of Theorem~\ref{theo:main_dcube}}\label{sec:proof_dcube}
Recall that $\Cube$ is the set of all discrete $d$-dimensional cubes and $\Cube_n$ is the set of all discrete $d$-dimensional cubes contained in $\{1,\ldots,n\}^d$. For $0\leq a<b\leq n$ let $\Cube_n(a,b)$ be the set of all discrete cubes $[\xx,\xx+h]_{\Z^d}\in \Cube_n$  such that $h\in [a,b]$. We write $l_n=[\log n]$ and $q_n=1/[\log n]$.
\begin{lemma}\label{lem:dcube1}
Fix some $\tau\in\R$, $0<a<b$, and let $u_n=u_n(\tau)$ be given by~\eqref{eq:def_ab_dcube}. Then,
$$
\lim_{n\to\infty}\P\left[\max_{A\in \Cube_n(al_n,bl_n)} \frac{\SSS(A)}{\sqrt{|A|}}\leq u_n\right]=
e^{-e^{-\tau}\frac 1{J_d}\int_{a}^{b}J_d(h)dh}.
$$
\end{lemma}
\begin{proof}
For $\kk\in\Z^d$ define $\DRR_{\kk,n}$ to be the set of all discrete cubes $[\xx,\xx+h]_{\Z^d}$  such that $h\in [al_n,bl_n]$ and  $x_i\in[k_il_n,(k_i+1)l_n]$ for all $1\leq i \leq d$.
Let also $\RR$ be the set of all (non-discrete) cubes $[\xx,\xx+h]$  such that $h\in [a,b]$ and  $x_i\in[0,1]$ for all $1\leq i \leq d$.
Let $E_{\kk,n}$ be the random event $\{\max_{A\in \DRR_{\kk,n}} \SSS(A)/\sqrt{|A|}>u_n\}$.
Recall that $\XXX$ is the standardized white noise. Using the affine invariance, we obtain
\begin{equation}\label{eq:eq1discr}
p_n
:=
\P\left[E_{\kk,n}\right ]
=
\P\left[\max_{A\in q_n\Z^{d+1}\cap \RR} \XXX(A)>u_n\right ].
\end{equation}
Note that $\kappa:=\lim_{n\to\infty}q_n u_n^2=2d$. Proposition~\ref{prop:dcube} and~\eqref{eq:def_ab_dcube} imply that
$$
p_n
\sim
\frac 1 {\sqrt{2\pi}}u_n^{2d+1}e^{-u_n^2/2} \left(\int_{\RR}J_d(h) d\xx dh \right)
\sim
e^{-\tau}\frac {l_n^d} {n^d} \frac 1{J_d}\int_{a}^{b}J_d(h)dh,
\;\;\; n\to\infty.
$$
The set $\Cube_n(al_n,bl_n)$ can be covered by approximately $n^d/l_n^d$ sets of the form $\DRR_{\kk,n}$. The statement of the lemma follows by applying Theorem~\ref{theo:poisson}. Its conditions can be verified in the same way as in the proof of Lemma~\ref{lem:ccube1}.
\end{proof}
The proof of Theorem~\ref{theo:main_dcube} can be completed as follows. Since $\Cube_n(al_n,bl_n)\subset \Cube_n$ for every $0<a<b$ and $n$ large enough, it follows from Lemma~\ref{lem:dcube1} that
\begin{equation}\label{eq:proof_dcube0}
\limsup_{n\to\infty}\P\left[\max_{A\in \Cube_n}\frac{\SSS(A)}{\sqrt{|A|}}\leq u_n\right]\leq e^{-e^{-\tau}\frac 1{J_d}\int_{a}^{b}J_d(h)dh}.
\end{equation}
Let us prove a converse inequality. The number of elements in the finite set $\Cube_n(0,al_n)$ does not exceed $al_nn^d$. Recall that the Gaussian tail probability $\bar \Phi$ satisfies $\bar \Phi(u)\leq C u^{-1}e^{-u^2/2}$, $u>0$. We obtain
\begin{equation}\label{eq:proof_dcube1}
\limsup_{n\to\infty}\P\left[\max_{A\in\Cube_n(0,al_n)}\frac{\SSS(A)}{\sqrt{|A|}}>u_n\right]
\leq
\limsup_{n\to\infty} C (al_n n^d)\cdot(u_n^{-1}e^{-u_n^2/2})
\leq  Ca,
\end{equation}
where the last inequality is a consequence of~\eqref{eq:def_ab_dcube}. Also, it follows from Lemma~\ref{lem:ccube2} that
\begin{align}
\limsup_{n\to\infty}\P\left[\max_{A\in \Cube_n(bl_n,n)}\frac{\SSS(A)}{\sqrt{|A|}}>u_n\right]
&\leq
\limsup_{n\to\infty}\P\left[\sup_{A\in \CCube_n(bl_n,n)}\XXX(A)>u_n\right] \label{eq:proof_dcube2}\\
&\leq \limsup_{n\to\infty} C(bl_n)^{-d} u_n^{2d+1}e^{-u_n^2/2}n^d\notag\\
&\leq Cb^{-d},\notag
\end{align}
where the last step follows from~\eqref{eq:def_ab_dcube}. It follows from~\eqref{eq:proof_dcube1}, \eqref{eq:proof_dcube2} and Lemma~\ref{lem:dcube1} that
\begin{equation}\label{eq:proof_dcube3}
\liminf_{n\to\infty}\P\left[\max_{A\in \Cube_n}\frac{\SSS(A)}{\sqrt{|A|}}\leq u_n\right]\geq e^{-e^{-\tau}\frac 1{J_d}\int_{a}^{b}J_d(h)dh}-C(a+b^{-d}).
\end{equation}
By Lemma~\ref{lem:JG_finite}, $J_d=\int_{0}^{\infty} J_d(h)dh$ is finite. Hence,  $\lim_{a\downarrow 0}\lim_{b\uparrow\infty}\int_{a}^{b} J_d(h)dh=J_d$. Letting $a\downarrow 0$ and $b\uparrow \infty$ in~\eqref{eq:proof_dcube0} and~\eqref{eq:proof_dcube3}, we obtain the statement of Theorem~\ref{theo:main_dcube}.

\subsection{Proof of Theorem~\ref{theo:main_drect}}
Recall that $\Rect$ is the set of all discrete $d$-dimensional rectangles and $\Rect_n$ is the set of all discrete $d$-dimensional rectangles contained in $\{1,\ldots,n\}^d$. For $0\leq a<b\leq n$ let $\Rect_n(a,b)$ be the set of all discrete rectangles $[\xx,\xx+\hh]_{\Z^d}\in \Rect_n$ such that  $h_i\in [a, b]$ for all $1\leq i\leq d$. Recall that we write $l_n=[\log n]$ and $q_n=1/[\log n]$.
\begin{lemma}\label{lem:drect1}
Fix some $\tau\in\R$, $0<a<b$, and let $u_n=u_n(\tau)$ be defined by~\eqref{eq:def_ab_drect}.  Then,
$$
\lim_{n\to\infty}\P\left[\max_{A\in \Rect_n(al_n,bl_n)} \frac{\SSS(A)}{\sqrt{|A|}}\leq u_n \right]=e^{-e^{-\tau}\left(\frac 1 {G_d} \int_{a}^b G_d(h)dh\right)^d}.
$$
\end{lemma}
\begin{proof}
For $\kk\in\Z^d$ let $\DRR_{\kk,n}$ be the set of discrete rectangles of the form $[\xx,\xx+\hh]_{\Z^d}$ such that $x_i\in[k_il_n, (k_i+1)l_n]$ and $h_i\in [al_n,bl_n]$ for every $1\leq i\leq d$. Let also $\RR$ be the set of all (non-discrete) rectangles of the form $[\xx,\xx+\hh]$, where $x_i\in[0, 1]$ and $h_i\in [a,b]$ for every $1\leq i\leq d$. Denote by $E_{\kk, n}$ the random event $\{\max_{A\in \DRR_{\kk,n}} \SSS(A)/\sqrt{|A|}>u_n\}$. Then, by the affine invariance
$$
p_n:=\P[E_{\kk,n}]=\P\left[\max_{A\in \RR\cap q_n\Z^{2d}} \XXX(A)>u_n\right].
$$
Note that $\kappa:=\lim_{n\to\infty}q_nu_n^2=2d$. By Proposition~\ref{prop:drect} and~\eqref{eq:def_ab_drect},
\begin{align}
p_n
&\sim
\frac{1}{\sqrt{2\pi}}u_n^{4d-1}e^{-u_n^2/2}\int_{\RR_{\kk}}\left(\prod_{i=1}^d G(y_i-x_i;2d)\right)d\xx d\hh\label{eq:eq2} \\
&=
e^{-\tau}\frac{l_n^d}{n^d} \left(\frac{1}{G_d}\int_{a}^b G_d(h)dh\right)^d,
\;\;\; n\to\infty.\notag
\end{align}
The set $\Rect_n(al_n,bl_n)$ can be covered by approximately $n^d/l_n^d$ sets of the form $\DRR_{\kk,n}$. To complete the proof, apply Theorem~\ref{theo:poisson} as in the proof of Lemma~\ref{lem:ccube1}.
\end{proof}

In the next lemma we estimate the high-crossing probability over the set of ``thin'' rectangles. Let $\Pcoll_n(a)$ be the set of all discrete rectangles $[\xx, \xx+\hh]_{\Z^d}\in \Rect_n$ such that $h_m\leq al_n$ for some $1\leq m\leq d$.
\begin{lemma}\label{lem:drect2}
Let $u_n$ be a sequence such that $u_n\sim c\sqrt{\log n}$ as $n\to\infty$ for some $c>0$.
Then,
\begin{equation}\label{eq:pr_drect2}
\P\left[\max_{A\in \Pcoll_n(a)} \frac{\SSS(A)}{\sqrt{|A|}}> u_n\right]
\leq C a u_n^{2d-1}e^{-u_n^2/2}n^{d}.
\end{equation}
\end{lemma}
\begin{proof}
We will prove~\eqref{eq:pr_drect2} by induction over the dimension $d\in\N$. If $d=1$, then~\eqref{eq:pr_drect2} follows from the Gaussian tail estimate $\bar \Phi(u_n)\leq Cu_n^{-1}e^{-u_n^2/2}$ and the fact that $|\Pcollone_n(a)|\leq al_nn$. Before proceeding further, let us show that~\eqref{eq:pr_drect2} implies that
\begin{align}
\P\left[\max_{A\in \Rect_n} \frac{\SSS(A)}{\sqrt{|A|}}> u_n\right]
&\leq C u_n^{2d-1}e^{-u_n^2/2}n^{d}.\label{eq:pr_drect1}
\end{align}
Let $\Rect_n(a,n)$ be the set of all discrete rectangles $[\xx,\xx+\hh]_{\Z^d}\in\Rect_n$ such that $h_i\geq a$ for all $1\leq i\leq d$, and recall that $\CRect_n(a,n)$ is the set of non-discrete rectangles $[\xx,\xx+\hh]\in\CRect_n$ such that $h_i\geq a$ for all $1\leq i\leq d$. Taking $a=b=l_n$ in Lemma~\ref{lem:crect2}, we obtain
\begin{equation*}
\P\left[\max_{A\in \Rect_n(l_n,n)} \frac{\SSS(A)}{\sqrt{|A|}}>u_n \right]
\leq
\P\left[\sup_{A\in \CRect_n(l_n,n)} \XXX(A)>u_n \right]
\leq
Cu_n^{2d-1}e^{-u_n^2/2}n^d.
\end{equation*}
Noting that $\Rect_n=\Pcoll_n(1)\cup \Rect_n(l_n,n)$, we see that~\eqref{eq:pr_drect2} implies~\eqref{eq:pr_drect1}.

Now, assume that~\eqref{eq:pr_drect2} and, consequently, \eqref{eq:pr_drect1} have been established in the $(d-1)$-dimensional setting. We may write $\Pcoll_n(a)=\cup_{m=1}^d \Pcoll_{n,m}(a)$, where $\Pcoll_{n,m}(a)$, $1\leq m\leq d$, is the set of all discrete rectangles $[\xx, \xx+\hh]_{\Z^d}\in \Rect_n$ such that $h_m\leq al_n$.
Consider the set $\Pcollm_{n,m}(x, h)$ of all discrete rectangles $[\xx,\yy]_{\Z^d}\in\Rect_n$ such that $x_m=x$ and $h_m=h$. The set $\Pcoll_{n,m}(a)$ can be written as a union of at most $al_n n$ sets of the form $\Pcollm_{n,m}(x, h)$. An easy inspection shows that as long as $x+h\leq n$, we have the following equality of laws  of random fields:
$$
\left\{\frac{\SSS(A)}{\sqrt{|A|}}, A\in \Pcollm_{n,m}(x, h)\right\}
\eqdistr
\left\{\frac{\SSS(A)}{\sqrt{|A|}}, A\in \RectDim_n\right\}.
$$
By the induction assumption, Eqn.~\eqref{eq:pr_drect1} holds in the $d-1$-dimensional setting. Hence,
$$
\P\left[\max_{A\in \Pcoll_{n,m}(a)} \frac{\SSS(A)}{\sqrt{|A|}}\leq u_n\right]
\leq
al_nn\cdot \P\left[\max_{A\in \RectDim_n} \frac{\SSS(A)}{\sqrt{|A|}}\leq u_n\right]
\leq
Cu_n^{2d-1}e^{-u_n^2/2}n^{d}.
$$
 Summing over $1\leq m\leq d$ establishes~\eqref{eq:pr_drect2} in the $d$-dimensional setting and completes the proof.
\end{proof}

We are now in position to complete the proof of Theorem~\ref{theo:main_drect}.
For every $0\leq a<b$, we have by Lemma~\ref{lem:drect1},
\begin{align}
\limsup_{n\to\infty}\P\left[\max_{A\in \Rect_n} \frac{\SSS(A)}{\sqrt{|A|}}\leq u_n \right]
&\leq
\lim_{n\to\infty}\P\left[\max_{A\in \Rect_n(al_n,bl_n)} \frac{\SSS(A)}{\sqrt{|A|}}\leq u_n \right]\label{eq:pr_drect3} \\
&=
e^{-e^{-\tau}\left(\frac 1{G_d}\int_{a}^b G_d(h)dh\right)^d}.\notag
\end{align}
We prove a converse inequality.  By Lemma~\ref{lem:crect2},
\begin{align}
\P\left[\max_{A\in \Rect_n(al_n,n)\backslash \Rect_n(al_n,bl_n)} \frac{\SSS(A)}{\sqrt{|A|}}> u_n \right]
&\leq
\P\left[\sup_{A\in \CRectm_n(al_n,bl_n)} \XXX(A)> u_n \right]\label{eq:pr_drect5}\\
&\leq
C b^{-1}a^{-(d-1)}l_n^{-d}u_n^{4d-1}e^{-u_n^2/2}n^d\notag\\
&\leq C b^{-1}a^{-(d-1)},\notag
\end{align}
where the last inequality follows from~\eqref{eq:def_ab_drect}.
Note that $\Rect_n\backslash \Rect_n(al_n,bl_n)=\Pcoll_n(a)\cup (\Rect_n(al_n,n)\backslash \Rect_n(al_n,bl_n))$.  Hence, Lemma~\ref{lem:drect2}  and~\eqref{eq:pr_drect5} imply that
\begin{equation}\label{eq:pr_drect6}
\P\left[\max_{A\in \Rect_n\backslash \Rect_n(al_n,bl_n)} \frac{\SSS(A)}{\sqrt{|A|}}> u_n \right]
\leq
 C (b^{-1}a^{-(d-1)}+a).
\end{equation}
It follows from Lemma~\ref{lem:drect1} and~\eqref{eq:pr_drect6} that
\begin{align}
\liminf_{n\to\infty}\P\left[\max_{A\in \Rect_n} \frac{\SSS(A)}{\sqrt{|A|}}\leq u_n \right]
\geq
e^{-e^{-\tau}\left(\frac 1 {G_d} \int_{a}^b G_d(h)dh\right)^d}-C (b^{-1}a^{-(d-1)}+a).\label{eq:pr_drect7}
\end{align}
Letting $a\downarrow 0$ and $b\uparrow \infty$ in~\eqref{eq:pr_drect3} and~\eqref{eq:pr_drect7} in such a way that $b^{-1}a^{-(d-1)}\to 0$, we complete the proof of Theorem~\ref{theo:main_drect}.

\section*{Acknowledgement}
The author is grateful to Axel Munk for useful discussions on the topic of the paper.

\bibliographystyle{plainnat}
\bibliography{paper1Abib}
\end{document}